\documentclass[12pt,bezier]{article}
\usepackage{times}
\usepackage{booktabs}
\usepackage{pifont}
\usepackage{floatrow}
\usepackage{caption}
\usepackage{mathrsfs}
\usepackage[fleqn]{amsmath}
\usepackage{amsfonts,amsthm,amssymb,mathrsfs,bbding}
\usepackage{txfonts}
\usepackage{graphics,multicol}
\usepackage{graphicx}
\usepackage{color}
\usepackage{enumerate}
\usepackage{tikz}
\usepackage{caption}
\usepackage{cite}
\usepackage{latexsym,bm}
\usepackage{indentfirst}
\usepackage{mathtools}
\evensidemargin=\oddsidemargin\topmargin=-1.5cm

\newtheorem{thm}{Theorem}[section]

\newtheorem{lem}{Lemma}[section]
\newtheorem{cor}{Corollary}[section]

\newtheorem{remark}{Remark}

\theoremstyle{definition}

\addtocounter{section}{0}

\begin{document}
\title{The  spectral radius of graphs with  fractional matching number\footnote{This work is supported
by  NSFC (No. 11371372).}}
\author{Qian-Qian Chen, \,\, Ji-Ming Guo\thanks{Corresponding author;
Email addresses:
 jimingguo@hotmail.com(J.-M. Guo).}\\
Department of Mathematics, \\
 East China University of Science and Technology, Shanghai, P. R. China.}

\date{}
\maketitle
{\flushleft\large\bf Abstract}
\vspace{0.1cm}

Let $\mathcal{G}_{n, \beta^*}$ $(\mathcal{G}^*_{n,\beta^*})$ be the set of all (connected) graphs of order $n$ with fractional matching number $\beta^*$.
In this paper, the graphs with maximal spectral radius in $\mathcal{G}_{n,\beta^*}$ and $\mathcal{G}^*_{n,\beta^*}$ are characterized, respectively.
Moreover, a lower bound for the spectral radius in graphs with order $n$ to guarantee the existence of a perfect fractional matching is also
given, which generalizes the main result of O [Suil O, Spectral radius and matchings in graphs, Linear Algebra and its Applications, 2020].

\begin{flushleft}
\end{flushleft}

\section{Introduction}\label{s-1}
Let $G=(V(G), E(G))$ be a simple graph with vertex set $V(G)=\{v_1, v_2, \cdots, v_n\}$ and edge set $E(G)$. The adjacency matrix of $G$ is $A(G)=(a_{ij})$,
where $a_{ij}=1$ if $v_i$ is adjacent to $v_j$ , and $0$ otherwise. The largest eigenvalue of $A(G)$ is called the spectral radius of $G$,
which denoted by $\rho(G)$. The positive unit eigenvector corresponding to $\rho(G)$ of a connected graph $G$
 is usually called the \emph{Perron vector} of $G$. The polynomial $det(xI-A(G))$ is called the \emph{characteristic polynomial} of $G$,
 where $I$ is the identity matrix.

Two distinct edges in a graph $G$ are \emph{independent} if they are not incident with a common vertex.
A set of pairwise independent edges in $G$ is called a \emph{matching} of $G$. A matching with maximum cardinality
is a \emph{maximum matching} of $G$, and the cardinality of a maximum matching is the \emph{matching number} of $G$,
denoted by $\beta(G)$. A \emph{perfect matching} in $G$ is a matching covering all vertices.
It is well known that $\beta(G)\leq\frac{n}{2}$ with equality hold if and only if $G$ has a prefect matching.
We call a matching in $G$ is a \emph{near perfect} if it covers all but one vertex of $G$.
A \emph{fractional matching} of a graph $G$ is a function $f$ giving each edge a number in $[0, 1]$ such
that $\sum_{e\in\Gamma(v)}f(e)\leq1$ for each $v\in V(G)$, where $\Gamma(v)$ is the set of edges incident to $v$.
The \emph{fractional matching number} of $G$, denoted by $\beta^*(G)$, is the maximum of
$\sum_{e\in E(G)}f(e)$
over all fractional matchings. Since $\sum_{e\in\Gamma(v)}f(e)\leq1$, it follows that
$\sum_{e \in E(G)} f(e)=\frac{1}{2} \sum_{v \in V(G)} \sum_{e\in \Gamma(v)} f(e) \leq n / 2,$
and so $\beta^{*}(G) \leq n / 2 .$ A \emph{fractional perfect matching} of a graph $G$ is a fractional
matching $f$ with $\sum_{e \in E(G)} f(e)=n / 2$, that is, $\beta^{*}(G)=n / 2 $. If a fractional
perfect matching takes only the values 0 or $1$, then it is a perfect matching.

Given a vertex subset $S$ of $G$, the subgraph induced by $S$ is denoted by $G[S]$. For any $v\in V(G)$, denoted by $N(v)$ the neighborhood
 of $v$. Let $G_1=(V_1,E_1)$ and $G_2=(V_2,E_2)$ be two vertex disjoint graphs.
 The \textit{union} $G_1 \cup G_2$ is a graph defined by $G_1 \cup G_2=(V_1 \cup V_2,E_1 \cup E_2)$,
 the \textit{join} $G_1 \vee G_2$  is a graph obtained from $G_1 \cup G_2$ by joining each vertex of $G_1$ to that of $G_2$.

 In 1986, Brualdi and Solheid \cite{Brualdi} proposed a classic question of extremal graph theory which has aroused the widely attention:
 \begin{center}
 What is the upper and lower bounds for the spectral radius
 in a family of graphs?
  \end{center}

The relationship between the matching number
 and the eigenvalues of graphs has been investigated by several researchers. For example,   Chang and Tian \cite{Chang}
 achieved the largest and the second largest spectral radius among unicyclic graphs with perfect matching. Guo \cite{Guo}
 presented a sharp upper bound for the Laplacian spectral radius of a tree in terms of the matching number and number of vertices.
 Brouwer and Haemers \cite{Brouwer} provided some sufficient conditions for the existence of a perfect matching
 in a graph in terms of the largest and the second smallest Laplacian eigenvalues. Recently, O \cite{SO} presented a lower bound for the adjacency
 spectral radius which guarantees the perfect matching in $G$. Up to now, much attention has been paid on this topic, and we refer to\cite{Cioa1,Cioa2,Cioa3}.

In addition, the relationship between  the fractional matching number and spectral radius of graphs were also investigated by several
researchers in recent years. O \cite{SO1} studied the connection between the
 spectral radius of a connected graph with minimum degree $\delta$ and its fractional matching number,
 and give a lower bound on the fractional matching number in terms of spectral radius and minimum degree.
 The relationship between the fractional matching number and laplacian spectral radius \cite{Xue},
 signless Laplacian spectral radius\cite{Pan} were also studied, respectively. Recently, Xue \cite{Xue} et al.
 gave the following sufficient condition for a graph to have a fractional perfect matching:
 \begin{thm}\label{thm-11}
 Let $G$ be a connected graph on $n$ vertices with minimum degree $\delta$. If $\lambda_1(G)<\delta\sqrt{\frac{n+1}{n-1}}$,
 then $G$ has a fractional perfect matching.
 \end{thm}

Let $\mathcal{G}_{n, \beta}$ be the set of all graphs of order $n$ with matching number $\beta$ and $\mathcal{G}^*_{n,\beta}$ be the set of all connected
graphs of order $n$ with matching number $\beta$. Feng et al. \cite{Feng} determined the graph maximizing the  spectral radius in $\mathcal{G}_{n,\beta}$,
which we present in Theorem \ref{thm-1}, and Chen\cite{Chen} characterized the graph with maximal spectral radius in $\mathcal{G}^*_{n,\beta}$,
which we present in Theorem \ref{thm-2}.

\begin{thm}\label{thm-1}
For any $G \in \mathcal{G}_{n,\beta}$, we have
\begin{enumerate}[(1)]
\vspace{-0.2cm}
\item if $n=2\beta$ or $2\beta+1$, then $\rho(G)\leq \rho(K_n)$, with equality if and only if $G\cong K_n$;
\vspace{-0.2cm}
\item if $2\beta+2\leq n<3\beta+2$, then $\rho(G)\leq2\beta$, with equality if and only if $G\cong K_{2\beta+1}\cup\overline{K_{n-2\beta-1}}$;
\vspace{-0.2cm}
\item if $n=3\beta+2$, then $\rho(G)\leq2\beta$, with equality if and only if $G\cong K_\beta\vee\overline{K_{n-\beta}}$ or $G\cong K_{2\beta+1}\cup\overline{K_{n-2\beta-1}}$;
\vspace{-0.4cm}
\item if $n>3\beta+2$, then $\rho(G)\leq\frac{\beta-1+\sqrt{(\beta-1)^{2}+4\beta(n-\beta)}}{2}$, with equality if and only if $G\cong K_\beta\vee\overline{K_{n-\beta}}$.

\end{enumerate}
\end{thm}
\begin{thm}\label{thm-2}
For any $G \in \mathcal{G}^*_{n,\beta}$, we have
\begin{enumerate}[(1)]
\vspace{-0.2cm}
\item if $n=2\beta$ or $2\beta+1$, then $\rho(G)\leq n-1$, with equality if and only if $G\cong K_n$;
\vspace{-0.2cm}
\item if $2\beta+2\leq n\leq3\beta-1$ , then $\rho(G)\leq \theta$, with equality if and only if $G\cong K_1\vee(K_{2\beta-1}\cup(n-2\beta)K_1)$, where $\theta$ is the largest root of $x^3-(2\beta-2)x+(1-n)x+2(\beta-1)(n-2\beta)=0$ and $\beta\geq 3, n\geq8$;
\vspace{-0.4cm}
\item if $n\geq3\beta$, then $\rho(G)\leq \frac{\beta-1+\sqrt{(\beta-1)^2+4\beta(n-\beta)}}{2}$, with equality if and only if $G\cong K_\beta\vee\overline{K_{n-\beta}}.$

\end{enumerate}
\end{thm}

Let $\mathcal{G}_{n, \beta^*}$ $(\mathcal{G}^*_{n,\beta^*})$ be the set of all (connected) graphs of order $n$ with fractional
matching number $\beta^*$. Inspired by the above observations, in this paper, we characterize the graph with maximal spectral radius in
 $\mathcal{G}_{n,\beta^*}$ and $\mathcal{G}^*_{n,\beta^*}$, respectively. Moreover, we give a lower bound for the spectral radius
of graphs with order $n$ to guarantee the existence of a perfect fractional matching, which generalizes the main result of O\cite{SO}.

\section{Preliminaries}\label{s-2}
Prior to showing our main results, in this section, we present some notations and some useful lemmas which will be used in the proofs of our main results.
\begin{lem}[\cite{Stevanovic}]\label{lem-1-1}
If $H$ is a proper subgraph of a connected graph $G$, then $$\rho(H)<\rho(G).$$
\end{lem}

Given a graph $G$, the vertex partition $\Pi: V(G)=V_1\cup V_2\cup \cdots \cup V_k$ is said to be an equitable partition
if for any $v\in V_i, |V_j\cap N(v)|=b_{ij}$ is constant depending only on $i, j(1\leq i, j\leq k)$. The matrix $Q(G)=(b_{ij})$ is called the \emph{quotient matrix} of $\Pi$.\

\begin{lem}[\cite{Godsil}]\label{lem-1-21}
Let $\Pi: V(G)=V_1\cup V_2\cup \cdots \cup V_k$ be an equitable partition of $G$ with quotient matrix $Q$.
Then $det(xI-Q(G))|det(xI-A(G))$. Furthermore, the largest eigenvalue of $Q(G)$ is just the spectral radius of $G$.
\end{lem}

\begin{lem}[\cite{ER}]\label{lem-1-2}
For any graph $G, 2 \beta^*(G)$ is an integer. Moreover, there is a fractional matching $f$ for which
$$
\sum_{e \in E(G)} f(e)=\beta^*(G),
$$
such that $f(e) \in\{0,1 / 2,1\}$ for every edge $e .$
\end{lem}

A fractional transversal of a graph $G$ is a function $g: V(G) \rightarrow[0,1]$ satisfying $\sum_{v \in e} g(v) \geq 1$ for every $e \in E(G) .$
The fractional transversal number is the infimum of $\sum_{v \in V(G)} g(v)$ taken over all fractional transversal $g$ of $G .$
By duality, the fractional transversal number of graph $G$ is the fractional matching number $\beta^*$ of $G$.
\begin{lem}[\cite{ER}]\label{lem-1-3}
For any graph $G$, there is a fractional transversal $g$ for which
$$
\sum_{v \in V(G)} g(v)=\beta^*(G),
$$such that $g(v) \in\{0,\frac{1}{2},1\}$ for every vertex.
\end{lem}

\begin{lem}[\cite{ER}]\label{lem-1-4}
The following are equivalent for a graph $G$.
\begin{enumerate}[(a)]
\vspace{-0.2cm}
\item $G$ has a fractional perfect matching.
\vspace{-0.2cm}
\item There is a partition $\{V_1,\ldots, V_k\}$ of the vertex set $V(G)$ such that, for each $i$, the graph $G[V_i]$ is either $K_2$ or Hamiltonian graph on an odd number of vertices.

\end{enumerate}
\end{lem}
\begin{remark}\label{rem-1}
From the proof of Lemma \ref{lem-1-4}, there is a fractional matching $f$,  such that, if  $G[V_i]\cong K_2$, let $f(e)=1$ where $e\in E(K_2)$. If $G[V_i]$ is a Hamiltonian graph on an odd number of vertices, let $f(e)=\frac{1}{2}$ for every edge $e$ in the Hamilton cycle, and $f(e)=0$, otherwise. It is obvious that $f$ is a fractional perfect matching.
\end{remark}
\section{The spectral radius with  fractional matching number }\label{s-4}

In this section, we first characterize the graph with maximal spectral radius in $\mathcal{G}_{n,\beta^*}$ and $\mathcal{G}^*_{n,\beta^*}$, respectively.
Next, we give a lower bound for the spectral radius of graphs with order $n$ to guarantee the existence of a perfect fractional matching.
Finally, we present a lower bound for the spectral radius which guarantees $\beta(G)\geq \beta+1$, where $\beta~(\beta\leq\frac{n-2}{2})$ is
a positive integer number, which generalizes the main result of O\cite{SO}.
\begin{lem}\label{lem-4-1}
Let $G$ be a  graph with the maximal spectral radius among the graphs with fractional matching number $\beta^*$, then there exist integer
numbers $t$ and $s$, such that $G\cong K_s\vee(K_{n-t-s}\cup tK_1)$, where $t\geq s\geq0$ and $t=n+s-2\beta^*$.
\end{lem}
\begin{proof}
By Lemma \ref{lem-1-3}, for any graph, there exists a fractional transversal number $g$ such that
$\sum_{v \in V(G)} g(v)=\beta^*$, where $g(v) \in\{0,1 / 2,1\}$ for every vertex $v$.
Define $$W=\{v\in V(G), g(v)=1\};$$  $$R=\{v\in V(G), g(v)=0\};$$ $$C=V(G)-W-R=\{v\in V(G), g(v)=\frac{1}{2}\}.$$
Assume that $|W|=s$, $|R|=t$. Note that, since $g$ is a fractional transversal, $\sum_{v \in e} g(v) \geq 1$, it is easy to see $G$ satisfies the following conditions:
\begin{enumerate}[(a)]
\vspace{-0.3cm}
\item $R$ is an independent set;
\vspace{-0.3cm}
\item there is no edge between $R$ and $C$;
\vspace{-0.3cm}
\item if $G$ is connected, then either $s=t=0$ , or $s,t\neq0$.
\end{enumerate}
 Besides,
\begin{equation}\label{equ-4-111}
\begin{array}{lll}
\beta^*&=&\frac{1}{2}|C|+|W|\\
&=&\frac{1}{2}(n-s-t)+s\\
&=&\frac{1}{2}(n-(t-s)).
\end{array}
\end{equation}
From equation (\ref{equ-4-111}), it follows that $t\geq s$ since $\beta^*(G)\leq\frac{n}{2}$.

Next, we claim $G\cong K_s\vee(K_{n-t-s}\cup tK_1)$. If not, let $\widetilde{G}$ be the graph obtained from $G$
by adding some edges such that $\widetilde{G}\cong K_s\vee(K_{n-t-s}\cup tK_1)$.
According to the minimality of the fractional transversal number, it follows that the fractional transversal number of $\widetilde{G}$ is still $\beta^*(G)$.
By duality, $\beta^*(\widetilde{G})=\beta^*(G)$ since the fractional transversal number equals to the fractional matching number.
And by Lemma \ref{lem-1-1}, it is easy to see $\rho(G)<\rho(\widetilde{G})$, a contradiction. Therefore, $G\cong K_s\vee(K_{n-t-s}\cup tK_1)$.

This completes the proof.
\end{proof}

\begin{thm}\label{thm-4-1}
For any $G\in\mathcal{G}^*_{n,\beta^*}$ , we have \begin{enumerate}[(i)]
\vspace{-0.2cm}
\item  If $n=2\beta^*$,  then $\rho(G)\leq n-1$, with equality if and only if $G\cong K_n$;
\vspace{-0.2cm}
\item If $2\beta^*+1\leq n<3\lceil\beta^*\rceil-3$, then $\rho(G)\leq \theta(n, \beta^*)$, with equality
if and only if $G\cong K_1\vee(K_{2\beta^*-2}\cup(n-2\beta^*+1)K_1)$, where $\theta(n, \beta^*)$ is the largest
root of $x^3-(2\beta^*-3)x^2-(n-1)x-4\beta^{*2}+2\beta^*n+8\beta^*-3n-3=0$;
\vspace{-0.2cm}
\item If $n\geq3\lceil\beta^*\rceil-3$,
then $\rho(G)\leq \frac{\lfloor\beta^*\rfloor-1+\sqrt{(\lfloor\beta^*\rfloor-1)^2+4\lfloor\beta^*\rfloor(n-\lfloor\beta^*\rfloor)}}{2}$,
with equality if and only if $G\cong K_{\lfloor\beta^*\rfloor}\vee\overline{K_{n-\lfloor\beta^*\rfloor}}.$

\end{enumerate}
\end{thm}
\begin{proof}
For convenience, Let $G_1=K_1\vee(K_{2\beta^*-2}\cup(n-2\beta^*+1)K_1)$ and $G_2=K_{\lfloor\beta^*\rfloor}\vee\overline{K_{n-\lfloor\beta^*\rfloor}}$, respectively.
It is easy to see that for any graph $G$, $\rho(G)\leq\rho(K_n)=n-1$, the equality holds if and only if $G=K_n$. Thus $(i)$ holds.
We now prove that $(ii)$ holds. Let $G^*$ be a graph with maximal spectral radius in $\mathcal{G}^*_{n,\beta^*}$.
Since $G^*$ is connected and $\beta^*(G)\leq\frac{n-1}{2}$, from Lemma \ref{lem-4-1}, there exists a positive integer number $s\ge 1$
such that $G^*\cong K_s\vee(K_{2\beta^*-2s}\cup (n+s-2\beta^*)K_1)$ and $\beta^*\geq s$.
If $s=1$, it is easy to see that $G^*=G_1$.
Considering the vertex partition $\Pi:V(G^*)=\{V(K_s), V(K_{2\beta^*-2s}), V(\overline{K_{n+s-2\beta^*}})\}$, the quotient matrix of $G^*$ is
\[Q(G^*):=\left(\begin{matrix}
s-1&2\beta^*-2s&n+s-2\beta^*\\
s&2\beta^*-2s-1&0\\
s&0&0\\
\end{matrix}\right).\]
The characteristic polynoimal of the matrix $Q(G^*)$ is
\begin{equation}\label{equ-5-1}
\begin{aligned}
f(x, s)=&x^3-(2\beta^*-s-2)x^2+(2\beta^* s-s^2-2\beta^*+s-sn+1)x-4\beta^{*2}s+2\beta^* ns+6\beta^* s^2\\&-2ns^2-2s^3+2s\beta^*-sn-s^2.
\end{aligned}
\end{equation}
Let $\theta_1\geq\theta_2\geq\theta_3$ be three roots of $f(x, s)=0$. By Lemma \ref{lem-1-21}, we have $\theta_1=\rho(G^*)$.
Since $K_{2\beta^*-1}$ is a subgraph of $G_1$, $\rho(G_1)>2\beta^*-2$ by Lemma \ref{lem-1-1}.
Next, we will prove  $\theta_2<2\beta^*-2<\rho(G_1)$ and $\theta_1\le 2\beta^*-2,\ (s\ge 2)$. In fact, let $D=diag(s, 2\beta^*-2s, n+s-2\beta^*)$, then
\[D^{\frac{1}{2}}Q(G^*)D^{-\frac{1}{2}}=\left(\begin{matrix}
s-1&\sqrt{s(2\beta^*-2s)}&\sqrt{s(n+s-2\beta^*)}\\
\sqrt{s(2\beta^*-2s)}&2\beta^*-2s-1&0\\
\sqrt{s(n+s-2\beta^*)}&0&0\\
\end{matrix}\right).\]
Since $D^{\frac{1}{2}}Q(G^*)D^{-\frac{1}{2}}$ and $Q(G^*)$ have the same eigenvalues, the Cauchy interlacing theorem implies that
 $\theta_2\leq 2\beta^*-2s-1<2\beta^*-2<\rho(G_1)$.  Thus, in order to prove $\theta_1\le 2\beta^*-2,\ (s\ge 2)$,
 we only need to prove $f(2\beta^*-2, s)\geq0,\ (s\ge 2)$.

If $\beta^*$ is not an integer number, by Lemma \ref{lem-1-2}, we know $\beta^*=\frac{k}{2}$, where $k$ is an odd number.
Thus $n<3\lceil\beta^*\rceil-3=3(\beta^*+\frac{1}{2})-3=3\beta^*-\frac{3}{2}$, we can get $n\leq3\beta^*-\frac{5}{2}$.
In the following, we need to consider the case of $2\beta^*+1\leq n\leq3\beta^*-\frac{5}{2}$
and $2\leq s\leq\beta^*-\frac{1}{2}$. According to Eq. (\ref{equ-5-1}), we have
\begin{equation}\label{equ-5-2}
\begin{array}{lll}
f(2\beta^*-2,s)&=(-2s^2+s)n-2s^3+(4\beta^*+1)s^2+(4{\beta^*}^2-8\beta^*+2)s\\&-4\beta^{*2}+6\beta^*-2\\
&\geq(4s-4){\beta^*}^2-(2s^2+5s-6)\beta^*+6s^2-\frac{s}{2}-2s^3-2  \quad (for~~n\leq3\beta^*-\frac{5}{2}).\\
\end{array}
\end{equation}
 Since $2s(4s-4)-2s^2-5s+6=6s^2-13s+6>0$ when $s\geq2$, we have
\[
-\frac{(-2s^2-5s+6)}{2(4s-4)}<s<s+\frac{1}{2},
\]
therefore, combining with $\beta^*\geq s+\frac{1}{2}$ and Eq. $(\ref{equ-5-2})$, we obtain
\begin{equation*}
\begin{array}{lll}
f(2\beta^*-2,s)&\geq&(4s-4){\beta^*}^2+(-2s^2-5s+6)\beta^*+6s^2-\frac{s}{2}-2s^3-2\\
&\geq&(4s-4)(s+\frac{1}{2})^2+(-2s^2-5s+6)(s+\frac{1}{2})+6s^2-\frac{s}{2}-2s^3-2\\
&=&0,
\end{array}
\end{equation*} the result follows.

If $\beta^*$ is an integer number, we need to consider the case of $2\beta^*+1\leq n\leq3\beta^*-4$ and $2\leq s\leq\beta^*$.
By similar reasoning as above, we have
\begin{equation*}
\begin{array}{lll}
f(2\beta^*-2,s)&=(-2s^2+s)n-2s^3+(4\beta^*+1)s^2+(4{\beta^*}^2-8\beta^*+2)s\\&-4\beta^{*2}+6\beta^*-2\\
&\geq(4s-4)\beta^{*2}-(2s^2+5s-6)\beta^*-2s^3+9s^2-2s-2 \quad (for~~n\leq3\beta^*-4)\\
&\geq4s-2 \quad (for~~s\leq\beta^*)\\
&>0.
\end{array}
\end{equation*}
Thus $(ii)$ holds.
\\
In the end, we prove that $(iii)$ holds. Considering the vertex partition $\{V(K_{\lfloor\beta^*\rfloor}), V(\overline{K_{n-\lfloor\beta^*\rfloor}})\}$ of $V(G_2)$, we have
\[Q(G_2):=\left(\begin{matrix}
\lfloor\beta^*\rfloor-1&n-\lfloor\beta^*\rfloor\\
\lfloor\beta^*\rfloor&0\\
\end{matrix}\right).\]
The characteristic polynoimal of the matrix $Q(G_2)$ is
\begin{equation}\label{equ-4-11}
g(x)=det(xI_2-Q(G_2))=x^2-(\lfloor\beta^*\rfloor-1)x-\lfloor\beta^*\rfloor(n-\lfloor\beta^*\rfloor).
\end{equation}
According to Lemma \ref{lem-1-21}, the largest root of $Q(G_2)$ is $\rho(G_2)=\frac{\lfloor\beta^*\rfloor-1+\sqrt{(\lfloor\beta^*\rfloor-1)^2+4\lfloor\beta^*\rfloor(n-\lfloor\beta^*\rfloor)}}{2}$ . Besides, $f(\rho(G_2), \lfloor\beta^*\rfloor)=0$.

If $\beta^*$ is an integer number, we  need to consider the case of $n\geq3\beta^*-3$ and $1\leq s\leq\beta^*-1$, since if $s=\lfloor\beta^*\rfloor$, $G^*\cong G_2$. Therefore, we have
\[
\begin{array}{lll}\label{equ-5-3}
\rho(G_2)&=&\frac{\beta^*-1+\sqrt{(\beta^*-1)^2+4\beta^*(n-\beta^*)}}{2}\\
&\geq&\frac{\beta^*-1+\sqrt{9\beta^{*2}-14\beta^*+1}}{2}  \quad (if~~n\geq3\beta^*-3)\\
&\geq&\frac{\beta^*-1+\sqrt{9\beta^{*2}-18\beta^*+9}}{2}  \quad (if~~\beta^*\geq2)\\
&=&2\beta^*-2.
\end{array}
\]
Next, we only need to prove $f(2\beta^*-2, s)\geq0$ when $1\leq s\leq\beta^*-1$. According to Eq.$(\ref{equ-5-1})$, we define
\begin{equation}\label{equ-1}
\begin{array}{lll}
&h(x):=f(x,s)-f(x,\beta^*)\\
=&x^3-(2\beta^*-s-2)x^2+(2\beta^* s-s^2-2\beta+s-sn+1)x-4\beta^*s+2\beta ns+6\beta s^2\\&-2ns^2-2s^3+2s\beta-sn-s^2-(x^3-(\beta-2)x^2+(\beta^{*2}-n\beta^*-\beta^*+1)x+\beta^*-n\beta^*)\\
=&(\beta^*-s)(-x^2+(n+s-\beta^*-1)x+2s^2+(2n-4\beta+1)s+n-\beta^*).\\
\end{array}
\end{equation}
By simple calculation, we have
\begin{equation*}
h(x)=(\beta^*-s)(-g(x)+(n+s-2\beta^*)x+2s^2+(2n-4\beta^*+1)s+(n-\beta^*)(1-\beta^*)).
\end{equation*}
Thus
\begin{equation*}
\begin{split}
&~~~~~f(\rho(G_2),s)\\
&=f(\rho(G_2),s)-f(\rho(G_2),\beta^*)=h(\rho(G_2))\\
&=(\beta^*-s)(-g(\rho(G_2))\!+\!(n\!+\!s\!-\!2\beta^*)\rho(G_2)+2s^2\!+\!(2n-4\beta^*+1)s\!+\!(n-\beta^*)(1\!-\!\beta^*))\\
&=(\beta^*-s)((n+s-2\beta^*)\rho(G_2)+2s^2+(2n-4\beta^*+1)s+(n-\beta^*)(1-\beta^*))\\
&\geq(\beta^*-s)((n+s-2\beta^*)(2\beta^*-2)+2s^2+(2n-4\beta^*+1)s+(n-\beta^*)(1-\beta^*))\\
&=(\beta^*-s)((\beta^*+2s-1)n+(s-2\beta^*)(2\beta^*-2)+2s^2+(1-4\beta^*)s-\beta^*(1-\beta^*))\\
&\geq(\beta^*-s)((4s-3)\beta^*+2s^2-7s+3) \quad (if~~n\geq3\beta^*-3)\\
&\geq(\beta^*-s)(6s^2-6s) \quad (if~~\beta^*\geq s+1)\\
&\geq0.
\end{split}
\end{equation*}
 Note that if $f(\rho(G_2),s)=0$, then all the inequalities become equalities in the above equation. Thus we have if $f(\rho(G_2),s)=0$,
 then $\beta^*=2$ and $s=1$. So, $n=3\beta^*-3=3$, a contradiction with $\beta^*\leq\frac{n}{2}$. Therefore, $f(\rho(G_2),s)>0$.

If $\beta^*$ is not an integer number, then $\lfloor\beta^*\rfloor=\beta^*-\frac{1}{2}$.
In this case, we need to consider $n\geq3\beta^*-\frac{3}{2}$ and $1\leq s\leq\beta^*-\frac{1}{2}$. By similar reasoning as above, we have
 \[
\begin{array}{lll}\label{equ-5-3}
\rho(G_2)&=&\frac{\lfloor\beta^*\rfloor-1+\sqrt{(\lfloor\beta^*\rfloor-1)^2+4\lfloor\beta^*\rfloor(n-\lfloor\beta^*\rfloor)}}{2}\\
&\geq&\frac{\beta^*-1+\sqrt{9\beta^{*2}-8\beta^*+1}}{2}  \quad (if~~n\geq3\beta^*-\frac{3}{2})\\
&>&2\beta^*-2.
\end{array}
\]
And,
\begin{equation}\label{equ-2}
\begin{array}{lll}
&h(x):=f(x,s)-f(x,\beta^*-\frac{1}{2})\\
&=(\beta^*-s-\frac{1}{2})(-g(x)+(n+s-2\beta^*)x+\beta^{*2}-(n+4s+1)\beta^*+(2s+\frac{1}{2})n+2s^2+\frac{1}{4}).
\end{array}
\end{equation}
Thus
\begin{equation*}
\begin{array}{lll}
&~~~~~f(\rho(G_2),s)\\
&=f(\rho(G_2),s)-f(\rho(G_2),\beta^*)=h(\rho(G_2))\\
&=(\beta^*-s-\frac{1}{2})(-g(\rho(G_2))+(n+s-2\beta^*)\rho(G_2)+\beta^{*2}-(n+4s+1)\beta^*+(2s+\frac{1}{2})n+2s^2+\frac{1}{4}\\
&\geq(\beta^*-s-\frac{1}{2})((n+s-2\beta^*)(2\beta^*-2)+\beta^{*2}-(n+4s+1)\beta^*+(2s+\frac{1}{2})n+2s^2+\frac{1}{4})\\
&=(\beta^*-s-\frac{1}{2})((2s+\beta^*-\frac{3}{2})n+2s^2-(2\beta^*+2)s-3\beta^{*2}+3\beta^*+\frac{1}{4})\\
&\geq(\beta^*-s-\frac{1}{2})((4s-3)\beta^*-5s+\frac{5}{2}+2s^2) \quad (if~~n\geq3\beta^*-\frac{3}{2})\\
&\geq(\beta^*-s-\frac{1}{2})(6s^2-6s+1) \quad (if~~\beta^*\geq s+\frac{1}{2})\\
&\geq0,
\end{array}
\end{equation*}
the equality holds if and only if $s=\beta^*-\frac{1}{2}=\lfloor\beta^*\rfloor$, we complete the
proof.
\end{proof}
Next, we characterize the graph with maximal spectral radius in $\mathcal{G}_{n,\beta^*}$ .
\begin{thm}\label{thm-4-2}
Let $G$ be a  graph in $\mathcal{G}_{n,\beta^*}$, then
\begin{enumerate}[(1)]
\vspace{-0.2cm}
\item  If $n=2\beta^*$,  then $\rho(G)\leq n-1$ with equality hold if and only if $G\cong K_n$.
\vspace{-0.2cm}
\item If $2\beta^*+1\leq n<3\lceil\beta^*\rceil-1$, then $\rho(G)\leq 2\beta^*$ with equality hold if and only if $G\cong K_{2\beta^*}\cup(n-2\beta^*)K_1$.
\vspace{-0.2cm}
\item If $n=3\lceil\beta^*\rceil-1$, then $\rho(G)\leq 2\beta^*$ with equality hold if and only if $G\cong K_{\lfloor\beta^*\rfloor}\vee\overline{K_{n-\lfloor\beta^*\rfloor}}$ or $K_{2\beta^*}\cup(n-2\beta^*)K_1$.
\vspace{-0.2cm}
\item If $n>3\lceil\beta^*\rceil-1$, then $\rho(G)\leq \frac{\lfloor\beta^*\rfloor-1+\sqrt{(\lfloor\beta^*\rfloor-1)^2+4\lfloor\beta^*\rfloor(n-\lfloor\beta^*\rfloor)}}{2}$ with equality hold if and only if $G\cong K_{\lfloor\beta^*\rfloor}\vee\overline{K_{n-\lfloor\beta^*\rfloor}}.$
\end{enumerate}
\end{thm}

\begin{proof}
For convenience, we denote  $G_3=K_{2\beta^*}\cup(n-2\beta^*)K_1$. It is easy to see $\rho(G_3)=2\beta^*-1$.

 Let $G^*$ be a graph with maximal spectral radius in $\mathcal{G}_{n,\beta^*}$. By  Lemma \ref{lem-4-1},  there exists an integer number $s$
 such that $G^*\cong K_s\vee(K_{2\beta^*-2s}\cup (n+s-2\beta^*)K_1)$, where $s\geq0$ .\\ $(1)$. Obviously. \\
$(2)$.
According to the proof of Theorem \ref{thm-4-1}, we have $\rho(G_3)=2\beta^*-1>\theta_2$, where
$\theta_2$ is the second largest root of the equation $f(x, s)=0$ (see Eq. (2)).

 Thus, we only need to prove $f(\rho(G_3),s)\!\geq\!0$. By Eq. (\ref{equ-5-1}), if $\beta^*$ is a fractional number, in this case, we need to consider
$2\beta^*+1\leq n<3\beta^*+\frac{1}{2}$,  $0\leq s\leq \beta^*-\frac{1}{2}$. Then we have
\begin{equation*}\label{equ-5-5}
\begin{array}{lll}
&f(\rho(G_3),s)=f(2\beta^*-1, s)\\
=&s(-2sn-2(-2\beta^{*2}-2\beta^*s+s^2+\beta^*)) \\
\geq&s(2\beta^*+s)(2\beta^*-2s-1) \quad (if~~n<3\beta^*+\frac{1}{2})\\
\geq&0 \quad (if~~0\leq s\leq\beta^*-\frac{1}{2}),
\end{array}
\end{equation*}
the equality holds if and only if  $s=0$.

If $\beta^*$ is an integer number, in this case, we need to consider
$2\beta^*+1\leq n<3\beta^*-1$,  $0\leq s\leq \beta^*$. As similar as above proof, by simple calculation, we have $f(\rho(G_3),s)=f(2\beta^*-1, s)\geq0$, and the equality holds if and only if  $s=0$,
 as required. \\
 $(3).$ If $s=\lfloor\beta^*\rfloor$, $G^*\cong G_2$, where $G_2=K_{\lfloor\beta^*\rfloor}\vee\overline{K_{n-\lfloor\beta^*\rfloor}}$ defined in Theorem \ref{thm-4-1}. In this case, by simple calculation,  we have
\[
\begin{array}{lll}\label{equ-5-3}
\rho(G_2)&=&\frac{\lfloor\beta^*\rfloor-1+\sqrt{(\lfloor\beta^*\rfloor-1)^2+4\lfloor\beta^*\rfloor(n-\lfloor\beta^*\rfloor)}}{2}\\
&=&2\beta^*-1=\rho(G_3),
\end{array}
\]
as similar as  the  proof of $(2)$, we have $f(\rho(G_2),s)=f(\rho(G_3),s)=f(2\beta^*-1, s)\geq0$, and the equality holds if and only if  $s=0$ or $s=\lfloor\beta^*\rfloor$, as required. \\
$(4).$  In this case, if $\beta^*$ is an integer number, we only need to consider the case of $n>3\beta^*-1$  and $0\leq s\leq\beta^*-1$. According to Eq. $(\ref{equ-5-1})$,
\[
\begin{array}{lll}\label{equ-5-3}
\rho(G_2)&=&\frac{\beta^*-1+\sqrt{(\beta^*-1)^2+4\beta^*(n-\beta^*)}}{2}\\
&>&\frac{\beta^*-1+\sqrt{9\beta^{*2}-6\beta^*+1}}{2}  \quad (if~~n>3\beta^*-1)\\
&=&2\beta^*-1.
\end{array}
\]
Therefore, according to the proof of Theorem \ref{thm-4-1}$(ii)$, we only need to prove $f(\rho(G_2), s)\geq0$. By similar reasoning as Theorem \ref{thm-4-1}$(iii)$, according to Eq. (\ref{equ-1}), we obtain
\begin{equation*}
\begin{array}{lll}
&f(\rho(G_2),s)\\
=&f(\rho(G_2),s)-f(\rho(G_2),\beta^*)=h(\rho(G_2))\\
=&(\beta^*-s)(-g(\rho(G_2))+(n+s-2\beta^*)\rho(G_2)+2s^2+(2n-4\beta^*+1)s+(n-\beta^*)(1-\beta^*))\\
=&(\beta^*-s)((n+s-2\beta^*)\rho(G_2)+2s^2+(2n-4\beta^*+1)s+(n-\beta^*)(1-\beta^*))\\
>&(\beta^*-s)((n+s-2\beta^*)(2\beta^*-1)+2s^2+(2n-4\beta^*+1)s+(n-\beta^*)(1-\beta^*))\\
>&2s(\beta^*-s)(2\beta^*+s-1) \quad (if~~n>3\beta^*-1)\\
\geq&0,
\end{array}
\end{equation*}
the equality holds if and only if $s=\beta^*$.

If $\beta^*$ is a fractional number, we only need to consider the case of $n>3\beta^*+\frac{1}{2}$  and $0\leq s\leq\beta^*-\frac{3}{2}$.
Thus, according to Eq. (\ref{equ-2}), we have
\begin{equation*}
\begin{array}{lll}
&~~~~~f(\rho(G_2),s)\\
&=f(\rho(G_2),s)-f(\rho(G_2),\beta^*)=h(\rho(G_2))\\
&=(\beta^*-s-\frac{1}{2})(-g(\rho(G_2))+(n+s-2\beta^*)\rho(G_2)+\beta^{*2}-(n+4s+1)\beta^*+(2s+\frac{1}{2})n+2s^2+\frac{1}{4}\\
&>(\beta^*-s-\frac{1}{2})((n+s-2\beta^*)(2\beta^*-1)+\beta^{*2}-(n+4s+1)\beta^*+(2s+\frac{1}{2})n+2s^2+\frac{1}{4})\\
&=(\beta^*-s-\frac{1}{2})((2s+\beta^*-\frac{1}{2})n+2s^2-(2\beta^*+1)s-3\beta^{*2}+\beta^*+\frac{1}{4})\\
&>(\beta^*-s-\frac{1}{2})(4\beta^*s+2s^2) \quad (if~~n>3\beta^*+\frac{1}{2})\\
&>0,
\end{array}
\end{equation*}
the equality holds if and only if $s=\beta^*-\frac{1}{2}=\lfloor\beta^*\rfloor$. The result follows.

\end{proof}

\begin{cor}\label{cor-5}
Let $G$ be a connected graph of order $n~(n\geq3)$. Given a number $\beta^*=\frac{k}{2}$, where $k$ is an integer number and $1\leq k\leq n-1$, then
\begin{enumerate}[(1)]
\vspace{-0.2cm}
\item When  $k$ is even, $\frac{n+3}{3}<\beta^* \leq\frac{n-1}{2}~(n\geq11)$, if $\rho(G)>\theta(n, \beta^*)$, where $\theta(n, \beta^*)$ is the largest root of $x^3-(2\beta^*-3)x^2-(n-1)x-4\beta^{*2}+2\beta^*n+8\beta^*-3n-3=0$, then $\beta^*(G)\geq \beta^*+\frac{1}{2}$.
\vspace{-0.2cm}
\item When $k$ is odd, $\frac{2n+3}{6}<\beta^* \leq\frac{n-1}{2}~(n\geq8~\text{and}~n\neq9)$, if $\rho(G)>\theta(n, \beta^*)$,  then $\beta^*(G)\geq \beta^*+\frac{1}{2}$.
\vspace{-0.2cm}
\item When $\lceil\beta^*\rceil\leq\frac{n+3}{3}$, if $\rho(G)> \frac{\lfloor\beta^*\rfloor-1+\sqrt{(\lfloor\beta^*\rfloor-1)^2+4\lfloor\beta^*\rfloor(n-\lfloor\beta^*\rfloor)}}{2}$, then $\beta^*(G)\geq \beta^*+\frac{1}{2}$.
\end{enumerate}
\end{cor}
\begin{proof}
If $G$ has a perfect fractional matching, then $\beta^*(G)=\frac{n}{2}=\frac{n-1}{2}+\frac{1}{2}\geq\beta^*+\frac{1}{2}$, as required.

If $G$ has no perfect fractional matching, we prove it by contradiction. If $\beta^*(G)\leq\beta^*$, then we can construct a graph $G'$,
where $G'$ is obtained from  $G$ by adding edges such that $\beta^*(G')=\beta^*$. It is obvious that $\rho(G')\geq\rho(G)$ by Lemma \ref{lem-1-1}. Combining with Theorem \ref{thm-4-1},
if $2\beta^*+1\leq n<3\lceil\beta^*\rceil-3$, then $\rho(G)\leq\rho(G')\leq \theta(n, \beta^*)$, a contradiction.
If $n\geq3\lceil\beta^*\rceil-3$, then $\rho(G)\leq\rho(G')\leq \frac{\lfloor\beta^*\rfloor-1+\sqrt{(\lfloor\beta^*\rfloor-1)^2+4\lfloor\beta^*\rfloor(n-\lfloor\beta^*\rfloor)}}{2}$, according to Theorem \ref{thm-4-1}, a contradiction.

The result follows.
\end{proof}

In the following, we give a lower bound for the spectral radius of a connected graph to guarantee the existence of a perfect fractional matching and  a perfect matching.
\begin{thm}\label{thm-6}
Let $G$ be a connected graph with order $n$.
\begin{enumerate}[(1)]
\vspace{-0.2cm}
\item  For $n\geq 8$ and $n\neq9$, if $\rho(G)>\theta(n)$, where $\theta(n)$ is the largest root of $x^3-(n-4)x^2-(n-1)x+2(n-4)=0$,
then $G$ has a perfect fractional matching; For $3\leq n\leq7$ or $n=9$, if $\rho(G)>\frac{\lfloor\beta^*\rfloor-1+\sqrt{(\lfloor\beta^*\rfloor-1)^2+4\lfloor\beta^*\rfloor(n-\lfloor\beta^*\rfloor)}}{2}$ where $\beta^*=\frac{n-1}{2}$, then $G$ has a perfect fractional matching.
\vspace{-0.2cm}
\item For $n\geq 8$ and $n$ is even, if $\rho(G)>\theta(n)$, $G$ also has a perfect matching.
\end{enumerate}
\end{thm}
\begin{proof}
By Corollary \ref{cor-5}, let $\beta^*=\frac{n-1}{2}$, and combining with the fact  $\beta^*(G)\leq\frac{n}{2}$, we have $(1)$ holds.
In the following, we prove that $(2)$ holds. For convenience, let $G_4=K_1\vee(K_{n-3}\cup 2K_1)$, by the proof of Theorem \ref{thm-4-1}, we know $\rho(G_4)=\theta(n)$.

From Lemma \ref{lem-1-2} and $(1)$, we know $G$ has a fractional perfect matching $f$
for which  $\sum_{e \in E(G)} f(e)=\frac{n}{2}$ such that $f(e)\in\{0, \frac{1}{2}, 1\}$ for every edge $e\in E(G)$.

Now we only need to show there exists some fractional perfect matching $f$  such that  $|\{e: \ f(e)=\frac{1}{2}\}|=0$. We prove it by contradiction.
If not, then for any fractional perfect matching of $G$, there are some edges $e\in E(G)$ with $f(e)=\frac{1}{2}$. Let $\mathcal{H}$ be the set of all connected graphs of even order $n$ with fractional perfect matching, for any $G\in\mathcal{H}$, $\rho(G)>\theta(n)$, and for any fractional perfect matching $f$ of $G$ with $\{e: \ f(e)=\frac{1}{2}\}\neq\emptyset$.
Assume $G^*$ is the graph with maximal spectral radius in $\mathcal{H}$. According to Lemma \ref{lem-1-4}
and Remark \ref{rem-1}, we know there exists a partition $\{V_1,\ldots, V_s\}$ of the vertex set $V(G^*)$ such that
 $G^*[V_i]$ $(1\leq i\leq k)$ is a Hamiltonian graph on an odd number of vertices, and $G^*[V_i]\cong K_2~(k+1\leq i\leq s)$. Besides, $f(e)=1$ for $e\in E(K_2)$;
 $f(e)=\frac{1}{2}$ for every edge $e$ in the Hamilton odd cycle and $f(e)=0$ otherwise. Note that, $n$ is an even number, thus $\beta^*(G^*)$ is an integer number and $k$ is even. Since $G^*$ has the maximal spectral radius, we claim
\begin{enumerate}[(a)]
\vspace{-0.2cm}
\item  $G^*[V_1]=G^*[V_2]=K_3$;
\vspace{-0.2cm}
\item the graph induced by $V(\cup_{i=3}^sG^*[V_i])$ is a complete graph.
\vspace{-0.2cm}
\end{enumerate}
If $(a)$ and $(b)$ do not hold, let $\widetilde{G}$ be the graph obtained from $G^*$ by adding edges such that $(a)$ and $(b)$ hold.
Then the resulting graph $\widetilde{G}$ has a partition $\{V_1, V_2, V_3\}$ such that $\widetilde{G}[V_1]=\widetilde{G}[V_2]=K_3$, $\widetilde{G}[V_3]=K_{n-6}$.
 Since $n-6$ is even, we can choose $\frac{n-6}{2}K_2$ from $\widetilde{G}[V_3]$ such that $f'(e)=1$;
 $f'(e)=\frac{1}{2}$ for every edge $e\in E(\widetilde{G}[V_1]\cup \widetilde{G}[V_2])$ and $f'(e)=0$ for other edges. It is easy to see
$\sum_{e \in E(\widetilde{G})} f'(e)=\frac{n}{2}$, and $\rho(\widetilde{G})>\rho(G^*)$, a contradiction. Thus, $G^*$ satisfies the conditions $(a)$ and $(b)$.
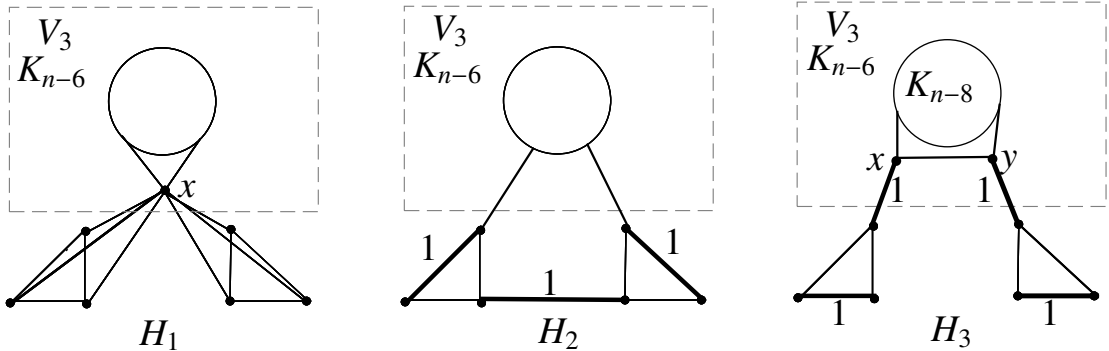
\begin{figure}[htbp]
\centering
\begin{tikzpicture}[x=1.00mm, y=1.00mm, inner xsep=0pt, inner ysep=0pt, outer xsep=0pt, outer ysep=0pt,scale=0.5]
\path[line width=0mm] (10.09,17.24) rectangle +(296.36,107.04);
\definecolor{L}{rgb}{0,0,0}
\definecolor{F}{rgb}{0,0,0}
\path[line width=0.18mm, draw=L, fill=F] (33.17,50.45) circle (1.00mm);
\path[line width=0.21mm, draw=L, fill=F] (54.26,61.37) circle (1.00mm);
\path[line width=0.18mm, draw=L, fill=F] (13.09,31.39) circle (1.00mm);
\path[line width=0.18mm, draw=L, fill=F] (33.42,31.13) circle (1.00mm);
\path[line width=0.21mm, draw=L] (54.26,61.37) -- (33.42,50.19);
\path[line width=0.21mm, draw=L] (54.26,61.12) -- (33.17,30.37);
\path[line width=0.30mm, draw=L] (53.75,61.37) -- (13.34,31.13);
\path[line width=0.18mm, draw=L] (13.34,31.89) -- (33.93,31.89);
\path[line width=0.18mm, draw=L, fill=F] (71.79,50.96) circle (1.00mm);
\path[line width=0.18mm, draw=L, fill=F] (71.54,31.89) circle (1.00mm);
\path[line width=0.18mm, draw=L, fill=F] (91.87,31.89) circle (1.00mm);
\path[line width=0.21mm, draw=L] (53.75,60.87) -- (71.54,50.70);
\path[line width=0.18mm, draw=L] (71.79,31.89) -- (93.14,31.89);
\path[line width=0.18mm, draw=L] (71.79,50.96) -- (92.38,31.89);
\path[line width=0.18mm, draw=L] (71.79,51.21) -- (71.54,31.64);
\path[line width=0.18mm, draw=L] (32.91,51.21) -- (32.91,30.12);
\path[line width=0.21mm, draw=L] (53.75,61.63) -- (71.03,32.66);
\path[line width=0.21mm, draw=L] (54.00,61.63) -- (92.13,31.39);
\path[line width=0.18mm, draw=L] (53.50,85.01) circle (14.25mm);
\path[line width=0.21mm, draw=L] (54.26,61.37) -- (63.66,75.10);
\path[line width=0.21mm, draw=L] (54.00,61.63) -- (42.82,75.86);
\path[line width=0.18mm, draw=L, fill=F] (33.17,50.45) circle (1.00mm);
\path[line width=0.21mm, draw=L, fill=F] (54.26,61.37) circle (1.00mm);
\path[line width=0.18mm, draw=L, fill=F] (13.09,31.39) circle (1.00mm);
\path[line width=0.18mm, draw=L, fill=F] (33.42,31.13) circle (1.00mm);
\path[line width=0.21mm, draw=L] (54.26,61.37) -- (33.42,50.19);
\path[line width=0.21mm, draw=L] (54.26,61.12) -- (33.17,30.37);
\path[line width=0.30mm, draw=L] (53.75,61.37) -- (13.34,31.13);
\path[line width=0.18mm, draw=L] (13.34,31.89) -- (33.93,31.89);
\path[line width=0.18mm, draw=L, fill=F] (71.79,50.96) circle (1.00mm);
\path[line width=0.18mm, draw=L, fill=F] (71.54,31.89) circle (1.00mm);
\path[line width=0.18mm, draw=L, fill=F] (91.87,31.89) circle (1.00mm);
\path[line width=0.21mm, draw=L] (53.75,60.87) -- (71.54,50.70);
\path[line width=0.18mm, draw=L] (71.79,31.89) -- (93.14,31.89);
\path[line width=0.18mm, draw=L] (71.79,50.96) -- (92.38,31.89);
\path[line width=0.18mm, draw=L] (71.79,51.21) -- (71.54,31.64);
\path[line width=0.18mm, draw=L] (32.91,51.21) -- (32.91,30.12);
\path[line width=0.21mm, draw=L] (53.75,61.63) -- (71.03,32.66);
\path[line width=0.21mm, draw=L] (54.00,61.63) -- (92.13,31.39);
\path[line width=0.18mm, draw=L] (53.50,85.01) circle (14.25mm);
\path[line width=0.21mm, draw=L] (54.26,61.37) -- (63.66,75.10);
\path[line width=0.21mm, draw=L] (54.00,61.63) -- (42.82,75.86);
\path[line width=0.18mm, draw=L, fill=F] (33.17,50.45) circle (1.00mm);
\path[line width=0.21mm, draw=L, fill=F] (54.26,61.37) circle (1.00mm);
\path[line width=0.18mm, draw=L, fill=F] (13.09,31.39) circle (1.00mm);
\path[line width=0.18mm, draw=L, fill=F] (33.42,31.13) circle (1.00mm);
\path[line width=0.21mm, draw=L] (54.26,61.37) -- (33.42,50.19);
\path[line width=0.21mm, draw=L] (54.26,61.12) -- (33.17,30.37);
\path[line width=0.30mm, draw=L] (53.75,61.37) -- (13.34,31.13);
\path[line width=0.18mm, draw=L] (13.34,31.89) -- (33.93,31.89);
\path[line width=0.18mm, draw=L, fill=F] (71.79,50.96) circle (1.00mm);
\path[line width=0.18mm, draw=L, fill=F] (71.54,31.89) circle (1.00mm);
\path[line width=0.18mm, draw=L, fill=F] (91.87,31.89) circle (1.00mm);
\path[line width=0.21mm, draw=L] (53.75,60.87) -- (71.54,50.70);
\path[line width=0.18mm, draw=L] (71.79,31.89) -- (93.14,31.89);
\path[line width=0.18mm, draw=L] (71.79,50.96) -- (92.38,31.89);
\path[line width=0.18mm, draw=L] (71.79,51.21) -- (71.54,31.64);
\path[line width=0.18mm, draw=L] (32.91,51.21) -- (32.91,30.12);
\path[line width=0.21mm, draw=L] (53.75,61.63) -- (71.03,32.66);
\path[line width=0.21mm, draw=L] (54.00,61.63) -- (92.13,31.39);
\path[line width=0.18mm, draw=L] (53.50,85.01) circle (14.25mm);
\path[line width=0.21mm, draw=L] (54.26,61.37) -- (63.66,75.10);
\path[line width=0.21mm, draw=L] (54.00,61.63) -- (42.82,75.86);
\path[line width=0.18mm, draw=L, fill=F] (33.17,50.45) circle (1.00mm);
\path[line width=0.21mm, draw=L, fill=F] (54.26,61.37) circle (1.00mm);
\path[line width=0.18mm, draw=L, fill=F] (13.09,31.39) circle (1.00mm);
\path[line width=0.18mm, draw=L, fill=F] (33.42,31.13) circle (1.00mm);
\path[line width=0.21mm, draw=L] (54.26,61.37) -- (33.42,50.19);
\path[line width=0.21mm, draw=L] (54.26,61.12) -- (33.17,30.37);
\path[line width=0.30mm, draw=L] (53.75,61.37) -- (13.34,31.13);
\path[line width=0.18mm, draw=L] (13.34,31.89) -- (33.93,31.89);
\path[line width=0.18mm, draw=L, fill=F] (71.79,50.96) circle (1.00mm);
\path[line width=0.18mm, draw=L, fill=F] (71.54,31.89) circle (1.00mm);
\path[line width=0.18mm, draw=L, fill=F] (91.87,31.89) circle (1.00mm);
\path[line width=0.21mm, draw=L] (53.75,60.87) -- (71.54,50.70);
\path[line width=0.18mm, draw=L] (71.79,31.89) -- (93.14,31.89);
\path[line width=0.18mm, draw=L] (71.79,50.96) -- (92.38,31.89);
\path[line width=0.18mm, draw=L] (71.79,51.21) -- (71.54,31.64);
\path[line width=0.18mm, draw=L] (32.91,51.21) -- (32.91,30.12);
\path[line width=0.21mm, draw=L] (53.75,61.63) -- (71.03,32.66);
\path[line width=0.21mm, draw=L] (54.00,61.63) -- (92.13,31.39);
\path[line width=0.18mm, draw=L] (53.50,85.01) circle (14.25mm);
\path[line width=0.21mm, draw=L] (54.26,61.37) -- (63.66,75.10);
\path[line width=0.21mm, draw=L] (54.00,61.63) -- (42.82,75.86);
\definecolor{L}{rgb}{0.502,0.502,0.502}
\path[line width=0.15mm, draw=L, dash pattern=on 2.00mm off 1.00mm] (12.87,55.68) rectangle +(82.13,54.37);
\draw(20.01,100.23) node[anchor=base west]{\fontsize{14.23}{17.07}\selectfont $V_3$};
\definecolor{L}{rgb}{0,0,0}
\path[line width=0.18mm, draw=L, fill=F] (138.17,50.72) circle (1.00mm);
\path[line width=0.18mm, draw=L, fill=F] (118.09,31.66) circle (1.00mm);
\path[line width=0.18mm, draw=L, fill=F] (138.42,31.41) circle (1.00mm);
\path[line width=0.18mm, draw=L] (138.67,51.74) -- (118.09,31.15);
\path[line width=0.18mm, draw=L] (118.34,32.17) -- (138.93,32.17);
\path[line width=0.18mm, draw=L, fill=F] (176.79,51.23) circle (1.00mm);
\path[line width=0.18mm, draw=L, fill=F] (176.54,32.17) circle (1.00mm);
\path[line width=0.18mm, draw=L, fill=F] (196.87,32.17) circle (1.00mm);
\path[line width=0.18mm, draw=L] (176.79,32.17) -- (198.14,32.17);
\path[line width=0.18mm, draw=L] (176.79,51.23) -- (197.38,32.17);
\path[line width=0.18mm, draw=L] (176.79,51.48) -- (176.54,31.92);
\path[line width=0.18mm, draw=L] (137.91,51.48) -- (137.91,30.39);
\path[line width=0.18mm, draw=L] (158.50,85.28) circle (14.25mm);
\path[line width=0.18mm, draw=L, fill=F] (138.17,50.72) circle (1.00mm);
\path[line width=0.18mm, draw=L, fill=F] (118.09,31.66) circle (1.00mm);
\path[line width=0.18mm, draw=L, fill=F] (138.42,31.41) circle (1.00mm);
\path[line width=0.18mm, draw=L] (138.67,51.74) -- (118.09,31.15);
\path[line width=0.18mm, draw=L] (118.34,32.17) -- (138.93,32.17);
\path[line width=0.18mm, draw=L, fill=F] (176.79,51.23) circle (1.00mm);
\path[line width=0.18mm, draw=L, fill=F] (176.54,32.17) circle (1.00mm);
\path[line width=0.18mm, draw=L, fill=F] (196.87,32.17) circle (1.00mm);
\path[line width=0.18mm, draw=L] (176.79,32.17) -- (198.14,32.17);
\path[line width=0.18mm, draw=L] (176.79,51.23) -- (197.38,32.17);
\path[line width=0.18mm, draw=L] (176.79,51.48) -- (176.54,31.92);
\path[line width=0.18mm, draw=L] (137.91,51.48) -- (137.91,30.39);
\path[line width=0.18mm, draw=L] (158.50,85.28) circle (14.25mm);
\path[line width=0.18mm, draw=L, fill=F] (138.17,50.72) circle (1.00mm);
\path[line width=0.18mm, draw=L, fill=F] (118.09,31.66) circle (1.00mm);
\path[line width=0.18mm, draw=L, fill=F] (138.42,31.41) circle (1.00mm);
\path[line width=0.18mm, draw=L] (138.67,51.74) -- (118.09,31.15);
\path[line width=0.18mm, draw=L] (118.34,32.17) -- (138.93,32.17);
\path[line width=0.18mm, draw=L, fill=F] (176.79,51.23) circle (1.00mm);
\path[line width=0.18mm, draw=L, fill=F] (176.54,32.17) circle (1.00mm);
\path[line width=0.18mm, draw=L, fill=F] (196.87,32.17) circle (1.00mm);
\path[line width=0.18mm, draw=L] (176.79,32.17) -- (198.14,32.17);
\path[line width=0.18mm, draw=L] (176.79,51.23) -- (197.38,32.17);
\path[line width=0.18mm, draw=L] (176.79,51.48) -- (176.54,31.92);
\path[line width=0.18mm, draw=L] (137.91,51.48) -- (137.91,30.39);
\path[line width=0.18mm, draw=L] (158.50,85.28) circle (14.25mm);
\path[line width=0.18mm, draw=L, fill=F] (138.17,50.72) circle (1.00mm);
\path[line width=0.18mm, draw=L, fill=F] (118.09,31.66) circle (1.00mm);
\path[line width=0.18mm, draw=L, fill=F] (138.42,31.41) circle (1.00mm);
\path[line width=0.60mm, draw=L] (138.67,51.74) -- (118.09,31.15);
\path[line width=0.18mm, draw=L] (118.34,32.17) -- (138.93,32.17);
\path[line width=0.18mm, draw=L, fill=F] (176.79,51.23) circle (1.00mm);
\path[line width=0.18mm, draw=L, fill=F] (176.54,32.17) circle (1.00mm);
\path[line width=0.18mm, draw=L, fill=F] (196.87,32.17) circle (1.00mm);
\path[line width=0.18mm, draw=L] (176.79,32.17) -- (198.14,32.17);
\path[line width=0.60mm, draw=L] (176.79,51.23) -- (197.38,32.17);
\path[line width=0.18mm, draw=L] (176.79,51.48) -- (176.54,31.92);
\path[line width=0.18mm, draw=L] (137.91,51.48) -- (137.91,30.39);
\path[line width=0.18mm, draw=L] (262.25,87.19) circle (14.25mm);
\definecolor{L}{rgb}{0.502,0.502,0.502}
\path[line width=0.15mm, draw=L, dash pattern=on 2.00mm off 1.00mm] (117.87,55.96) rectangle +(82.13,54.37);
\draw(125.01,101.33) node[anchor=base west]{\fontsize{14.23}{17.07}\selectfont $V_3$};
\definecolor{L}{rgb}{0,0,0}
\path[line width=0.60mm, draw=L] (137.88,32.47) -- (176.91,32.20);
\draw(14.66,90.47) node[anchor=base west]{\fontsize{14.23}{17.07}\selectfont $K_{n-6}$};
\draw(119.93,91.57) node[anchor=base west]{\fontsize{14.23}{17.07}\selectfont $K_{n-6}$};
\draw(47.18,20.32) node[anchor=base west]{\fontsize{14.23}{17.07}\selectfont $H_1$};
\draw(153.28,20.60) node[anchor=base west]{\fontsize{14.23}{17.07}\selectfont $H_2$};
\path[line width=0.30mm, draw=L] (208.59,122.28);
\path[line width=0.18mm, draw=L, fill=F] (242.62,51.82) circle (1.00mm);
\path[line width=0.18mm, draw=L, fill=F] (222.54,32.76) circle (1.00mm);
\path[line width=0.18mm, draw=L, fill=F] (242.87,32.51) circle (1.00mm);
\path[line width=0.18mm, draw=L] (243.12,52.84) -- (222.54,32.25);
\path[line width=0.18mm, draw=L] (222.79,33.27) -- (243.38,33.27);
\path[line width=0.18mm, draw=L, fill=F] (281.24,52.33) circle (1.00mm);
\path[line width=0.18mm, draw=L, fill=F] (280.99,33.27) circle (1.00mm);
\path[line width=0.18mm, draw=L, fill=F] (301.32,33.27) circle (1.00mm);
\path[line width=0.18mm, draw=L] (281.24,33.27) -- (302.59,33.27);
\path[line width=0.18mm, draw=L] (281.24,52.33) -- (301.83,33.27);
\path[line width=0.18mm, draw=L] (281.24,52.58) -- (280.99,33.01);
\path[line width=0.18mm, draw=L] (242.36,52.58) -- (242.36,31.49);
\path[line width=0.18mm, draw=L, fill=F] (242.62,51.82) circle (1.00mm);
\path[line width=0.18mm, draw=L, fill=F] (222.54,32.76) circle (1.00mm);
\path[line width=0.18mm, draw=L, fill=F] (242.87,32.51) circle (1.00mm);
\path[line width=0.18mm, draw=L] (243.12,52.84) -- (222.54,32.25);
\path[line width=0.18mm, draw=L] (222.79,33.27) -- (243.38,33.27);
\path[line width=0.18mm, draw=L, fill=F] (281.24,52.33) circle (1.00mm);
\path[line width=0.18mm, draw=L, fill=F] (280.99,33.27) circle (1.00mm);
\path[line width=0.18mm, draw=L, fill=F] (301.32,33.27) circle (1.00mm);
\path[line width=0.18mm, draw=L] (281.24,33.27) -- (302.59,33.27);
\path[line width=0.18mm, draw=L] (281.24,52.33) -- (301.83,33.27);
\path[line width=0.18mm, draw=L] (281.24,52.58) -- (280.99,33.01);
\path[line width=0.18mm, draw=L] (242.36,52.58) -- (242.36,31.49);
\path[line width=0.18mm, draw=L, fill=F] (242.62,51.82) circle (1.00mm);
\path[line width=0.18mm, draw=L, fill=F] (222.54,32.76) circle (1.00mm);
\path[line width=0.18mm, draw=L, fill=F] (242.87,32.51) circle (1.00mm);
\path[line width=0.18mm, draw=L] (243.12,52.84) -- (222.54,32.25);
\path[line width=0.18mm, draw=L] (222.79,33.27) -- (243.38,33.27);
\path[line width=0.18mm, draw=L, fill=F] (281.24,52.33) circle (1.00mm);
\path[line width=0.18mm, draw=L, fill=F] (280.99,33.27) circle (1.00mm);
\path[line width=0.18mm, draw=L, fill=F] (301.32,33.27) circle (1.00mm);
\path[line width=0.18mm, draw=L] (281.24,33.27) -- (302.59,33.27);
\path[line width=0.18mm, draw=L] (281.24,52.33) -- (301.83,33.27);
\path[line width=0.18mm, draw=L] (281.24,52.58) -- (280.99,33.01);
\path[line width=0.18mm, draw=L] (242.36,52.58) -- (242.36,31.49);
\path[line width=0.18mm, draw=L, fill=F] (242.62,51.82) circle (1.00mm);
\path[line width=0.18mm, draw=L, fill=F] (222.54,32.76) circle (1.00mm);
\path[line width=0.18mm, draw=L, fill=F] (242.87,32.51) circle (1.00mm);
\path[line width=0.30mm, draw=L] (243.12,52.84) -- (222.54,32.25);
\path[line width=0.60mm, draw=L] (222.79,33.27) -- (243.38,33.27);
\path[line width=0.18mm, draw=L, fill=F] (281.24,52.33) circle (1.00mm);
\path[line width=0.18mm, draw=L, fill=F] (280.99,33.27) circle (1.00mm);
\path[line width=0.18mm, draw=L, fill=F] (301.32,33.27) circle (1.00mm);
\path[line width=0.66mm, draw=L] (281.24,33.27) -- (302.59,33.27);
\path[line width=0.30mm, draw=L] (281.24,52.33) -- (301.83,33.27);
\path[line width=0.18mm, draw=L] (281.24,52.58) -- (280.99,33.01);
\path[line width=0.18mm, draw=L] (242.36,52.58) -- (242.36,31.49);
\definecolor{L}{rgb}{0.502,0.502,0.502}
\path[line width=0.15mm, draw=L, dash pattern=on 2.00mm off 1.00mm] (222.32,57.06) rectangle +(82.13,54.37);
\draw(229.46,102.43) node[anchor=base west]{\fontsize{14.23}{17.07}\selectfont $V_3$};
\draw(224.66,93.48) node[anchor=base west]{\fontsize{14.23}{17.07}\selectfont $K_{n-6}$};
\draw(257.73,21.16) node[anchor=base west]{\fontsize{14.23}{17.07}\selectfont $H_3$};
\definecolor{L}{rgb}{0,0,0}
\path[line width=0.60mm, draw=L] (248.68,69.71) -- (241.85,51.09);
\draw(57.41,59.61) node[anchor=base west]{\fontsize{14.23}{17.07}\selectfont $x$};
\path[line width=0.30mm, draw=L] (28.00,45.60) -- (28.27,45.32);
\path[line width=0.30mm, draw=L] (32.67,50.54) -- (13.70,31.85);
\path[line width=0.66mm, draw=L] (274.00,69.98) -- (281.35,51.73);
\path[line width=0.30mm, draw=L] (138.66,51.46) -- (152.28,72.70);
\path[line width=0.30mm, draw=L] (177.33,51.73) -- (166.44,73.79);
\path[line width=0.30mm, draw=L, fill=F] (248.68,69.16) circle (1.00mm);
\path[line width=0.30mm, draw=L, fill=F] (274.27,69.71) circle (1.00mm);
\path[line width=0.30mm, draw=L] (248.40,69.98) -- (274.55,70.25);
\path[line width=0.30mm, draw=L] (248.95,70.25) -- (248.95,82.51);
\path[line width=0.30mm, draw=L] (274.55,70.25) -- (276.18,84.41);
\draw(240.51,66.71) node[anchor=base west]{\fontsize{14.23}{17.07}\selectfont $x$};
\draw(276.45,67.53) node[anchor=base west]{\fontsize{14.23}{17.07}\selectfont $y$};
\draw(250.80,86.40) node[anchor=base west]{\fontsize{14.23}{17.07}\selectfont $K_{n-8}$};
\draw(121.78,42.75) node[anchor=base west]{\fontsize{14.23}{17.07}\selectfont $1$};
\draw(154.18,33.49) node[anchor=base west]{\fontsize{14.23}{17.07}\selectfont $1$};
\draw(186.86,44.38) node[anchor=base west]{\fontsize{14.23}{17.07}\selectfont $1$};
\draw(246.50,58.27) node[anchor=base west]{\fontsize{14.23}{17.07}\selectfont $1$};
\draw(269.64,58.27) node[anchor=base west]{\fontsize{14.23}{17.07}\selectfont $1$};
\draw(230.98,25.5) node[anchor=base west]{\fontsize{14.23}{17.07}\selectfont $1$};
\draw(287.07,25.5) node[anchor=base west]{\fontsize{14.23}{17.07}\selectfont $1$};
\end{tikzpicture}%
\caption{ $H_1,~H_2~~and~~H_3$.}\label{fig-1}
\end{figure}

From above, we know $G^*$ has a partition $\{V_1, V_2, V_3\}$
such that $G^*[V_1]=G^*[V_2]=K_3$, $G^*[V_3]=K_{n-6}$, next, we claim there is no edge between $G^*[V_1]$ and $G^*[V_2]$. If not, since $G^*[V_3]=K_{n-6}$, we can choose $\frac{n-6}{2}K_2$ as matching edges from $K_{n-6}$ such that $f''(e)=1$, and by replacing weight $\frac{1}{2}$ on thick edges with $1$, and other edges with $f''(e)=0$ (see Fig. 1 $H_2$),
 we find a fractional perfect matching $f''$ with  $|\{e: f''(e)=\frac{1}{2}\}|=0$, a contradiction.
Let $H_1 \cong K_1\vee(2K_3\cup K_{n-7})$  (see Fig.1).
 If $G^*\ncong H_1$, then $G^*$ must have the graph $H_3$ (see Fig.1) as a spanning subgraph.
Since $G^*[V_3\backslash\{x,y\}]=K_{n-8}$, we can choose $\frac{n-8}{2}K_2$ as matching edges from $K_{n-8}$ such that $f^*(e)=1$, and by replacing weight $\frac{1}{2}$ on thick edges with $1$, and other edges with $f^*(e)=0$ (see Fig. 1 $H_3$), we find a fractional perfect matching $f^*$  with $|\{e: f^*(e)=\frac{1}{2}\}|=0$,
a contradiction. From the maximality of $\rho(G^*)$, we have $G^*\cong H_1$.

If $n=8$, $G^*\cong K_1\vee(2K_3\cup K_1)$, by a simple calculation, $3.73\approx\rho(G^*)<\rho(K_1\vee(K_5\cup2K_1))=\theta(8)\approx5.07$, a contradiction.

If $n\geq 10$, $|V_3\backslash x|=n-7\geq 3$, let $X$ be the Perron vector of $G^*$ with respect to $\rho(G^*)=\rho$. From the symmetry of $G^*$, we can assume that $X$ take the same value (say $a, b, c$) on the vertices of $V_3\backslash x, x, V_1\cup V_2$, respectively. Then from $A(G^*)X=\rho X$, we have
\begin{equation*}
\begin{aligned}
 \rho a&=(n-8)a+b\\
 \rho c&=2c+b\\
\end{aligned}
\end{equation*}
Thus, $c=\frac{b}{\rho-2}\leq\frac{b}{\rho-n+8}=a$ since $\rho>\rho(K_{n-6})=n-7\geq3$. Now, assume $w\in V_1\cup V_2$, let $G'$ be the graph obtained from $G^*$ by removing all edges incident with $w$  and making $w$ adjacents to all vertices of $V_3$. Using the Rayleigh quotient, we have
$$\rho(G')-\rho(G^*)\geq X^TA(G')X-X^TA(G^*)X=2c((n-7)a-2c)>0.$$
Therefore, same as $w$, by moving two vertices of $V_1$ into $V_3$ in $G^*$, and two vertices of $V_2$ into $V_3$ in $G^*$, the resulting graph $G''\cong K_1\vee(K_{n-3}\cup 2K_1)=G_4$, and $\theta(n)=\rho(G'')>\rho(G^*)$, a contradiction.

From above discussion, we know there is no graph with maximal spectral radius in $\mathcal{H}$, which means $\mathcal{H}=\emptyset$. Therefore, $|\{e: \ f(e)=\frac{1}{2}\}|=0$,  we complete the proof.
\end{proof}

From Theorem \ref{thm-6}, we can get the following  result of O \cite{SO}.
\begin{cor}[\cite{SO}]\label{cor3-2}
Let $n\geq8$ be an even integer or $n=4$. If $G$ is an $n$-vertex connected graph with $\rho(G)>\theta(n)$, where $\theta(n)$ is the largest root of $x^3-(n-4)x^2-(n-1)x+2(n-4)=0$, then $G$ has a perfect matching. For $n=6$, if $\rho(G)>\frac{1+\sqrt{33}}{2}$, then $G$ has a perfect matching.
\end{cor}
\begin{proof}

According to Theorem \ref{thm-6}, if $n\geq8$, the result follows.

If $n=4$, by Theorem \ref{thm-6} $(1)$, let $\beta^*=\frac{3}{2}$, if $\rho(G)>\frac{\lfloor\frac{3}{2}\rfloor-1+\sqrt{(\lfloor\frac{3}{2}\rfloor-1)^2+4\lfloor\frac{3}{2}\rfloor(4-\lfloor\frac{3}{2}\rfloor)}}{2}=\sqrt{3}$ , then $G$ has a perfect fractional matching. Besides $\sqrt{3}$ is the largest root of $x^3-3x=0$. Let $f$ be a fractional matching with the greatest number of edges $e$ with $f(e)=1$,it is easy to see $G$ have no edges with weight $\frac{1}{2}$ since $\beta^*(G)=\frac{n}{2}=2$, as required.
\begin{figure}[htbp]
\centering
\begin{tikzpicture}[x=1.00mm, y=1.00mm, inner xsep=0pt, inner ysep=0pt, outer xsep=0pt, outer ysep=0pt,scale=0.7]
\path[line width=0mm] (34.71,47.61) rectangle +(193.86,40.45);
\definecolor{L}{rgb}{0,0,0}
\definecolor{F}{rgb}{0,0,0}
\path[line width=0.15mm, draw=L, fill=F] (69.69,71.76) circle (1.00mm);
\path[line width=0.15mm, draw=L, fill=F] (53.61,82.56) circle (1.00mm);
\path[line width=0.15mm, draw=L, fill=F] (53.89,61.78) circle (1.00mm);
\path[line width=0.15mm, draw=L, fill=F] (104.33,71.20) circle (1.00mm);
\path[line width=0.15mm, draw=L, fill=F] (115.97,85.06) circle (1.00mm);
\path[line width=0.15mm, draw=L, fill=F] (117.08,61.78) circle (1.00mm);
\path[line width=0.15mm, draw=L] (53.34,83.67) -- (53.61,61.22);
\path[line width=0.15mm, draw=L] (53.06,82.29) -- (69.69,71.48);
\path[line width=0.15mm, draw=L] (53.34,60.95) -- (70.52,72.03);
\path[line width=0.15mm, draw=L] (104.05,71.76) -- (116.52,85.89);
\path[line width=0.15mm, draw=L] (104.05,70.92) -- (117.08,61.78);
\path[line width=0.15mm, draw=L] (115.97,85.61) -- (117.63,60.67);
\path[line width=0.15mm, draw=L] (69.69,71.20) -- (105.44,71.20);
\path[line width=0.15mm, draw=L, fill=F] (160.04,85.06) circle (1.00mm);
\path[line width=0.15mm, draw=L, fill=F] (159.76,61.50) circle (1.00mm);
\path[line width=0.15mm, draw=L, fill=F] (175.00,72.03) circle (1.00mm);
\path[line width=0.15mm, draw=L, fill=F] (202.72,72.31) circle (1.00mm);
\path[line width=0.15mm, draw=L, fill=F] (218.79,85.06) circle (1.00mm);
\path[line width=0.15mm, draw=L, fill=F] (218.79,62.06) circle (1.00mm);
\path[line width=0.15mm, draw=L] (175.28,71.76) -- (159.48,85.89);
\path[line width=0.15mm, draw=L] (175.28,71.76) -- (158.93,61.22);
\path[line width=0.15mm, draw=L] (174.45,71.48) -- (202.72,72.03);
\path[line width=0.15mm, draw=L] (202.72,72.31) -- (218.79,85.61);
\path[line width=0.15mm, draw=L] (202.16,72.03) -- (219.34,61.50);
\path[line width=0.15mm, draw=L] (219.34,85.06) -- (219.07,60.95);
\path[line width=0.15mm, draw=L] (159.76,85.06) -- (159.48,60.95);
\definecolor{T}{rgb}{0,0,0}
\draw[T] (108,60) node[anchor=base west]{\fontsize{14.23}{17.07}\selectfont \textcolor[rgb]{0, 0, 0}{$\frac{1}{2}$}};
\draw[T] (60,83) node[anchor=base west]{\fontsize{14.23}{17.07}\selectfont \textcolor[rgb]{0, 0, 0}{$\frac{1}{2}$}};
\draw[T] (60,60) node[anchor=base west]{\fontsize{14.23}{17.07}\selectfont \textcolor[rgb]{0, 0, 0}{$\frac{1}{2}$}};
\draw[T] (108,83) node[anchor=base west]{\fontsize{14.23}{17.07}\selectfont \textcolor[rgb]{0, 0, 0}{$\frac{1}{2}$}};
\draw[T] (118.54,70) node[anchor=base west]{\fontsize{14.23}{17.07}\selectfont \textcolor[rgb]{0, 0, 0}{$\frac{1}{2}$}};
\draw[T] (48,70) node[anchor=base west]{\fontsize{14.23}{17.07}\selectfont \textcolor[rgb]{0, 0, 0}{$\frac{1}{2}$}};
\draw[T] (85.21,72.59) node[anchor=base west]{\fontsize{14.23}{17.07}\selectfont \textcolor[rgb]{0, 0, 0}{$0$}};
\draw[T] (152.83,71.20) node[anchor=base west]{\fontsize{14.23}{17.07}\selectfont \textcolor[rgb]{0, 0, 0}{$1$}};
\draw[T] (219.07,71.48) node[anchor=base west]{\fontsize{14.23}{17.07}\selectfont \textcolor[rgb]{0, 0, 0}{$1$}};
\draw[T] (183.04,72.31) node[anchor=base west]{\fontsize{14.23}{17.07}\selectfont \textcolor[rgb]{0, 0, 0}{$1$}};
\draw[T] (166.69,77.58) node[anchor=base west]{\fontsize{14.23}{17.07}\selectfont \textcolor[rgb]{0, 0, 0}{$0$}};
\draw[T] (166.69,77.58) node[anchor=base west]{\fontsize{14.23}{17.07}\selectfont \textcolor[rgb]{0, 0, 0}{$0$}};
\draw[T] (202.44,62) node[anchor=base west]{\fontsize{14.23}{17.07}\selectfont \textcolor[rgb]{0, 0, 0}{$0$}};
\draw[T] (202.16,78.13) node[anchor=base west]{\fontsize{14.23}{17.07}\selectfont \textcolor[rgb]{0, 0, 0}{$0$}};
\draw[T] (167.24,62) node[anchor=base west]{\fontsize{14.23}{17.07}\selectfont \textcolor[rgb]{0, 0, 0}{$0$}};
\draw[T] (79.39,50.69) node[anchor=base west]{\fontsize{14.23}{17.07}\selectfont \textcolor[rgb]{0, 0, 0}{$(a)$}};
\draw[T] (183.87,52.63) node[anchor=base west]{\fontsize{14.23}{17.07}\selectfont \textcolor[rgb]{0, 0, 0}{$(b)$}};
\end{tikzpicture}%
\caption{the case of n=6.}\label{fig-1}
\end{figure}
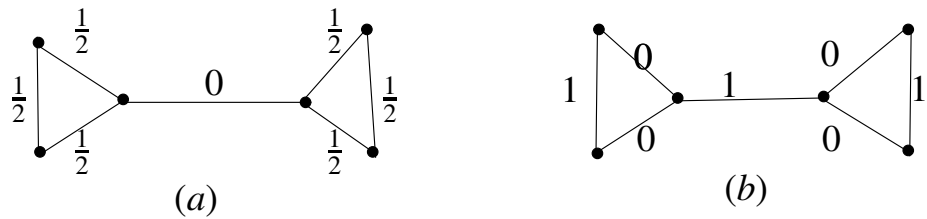

If $n=6$, by Theorem \ref{thm-6} $(1)$, let $\beta^*=\frac{5}{2}$, if $\rho(G)>\frac{\lfloor\frac{5}{2}\rfloor-1+\sqrt{(\lfloor\frac{5}{2}\rfloor-1)^2+4\lfloor\frac{5}{2}\rfloor(6-\lfloor\frac{5}{2}\rfloor)}}{2}=\frac{1+\sqrt{33}}{2},$ then $G$ has a perfect fractional matching. Let $f$ be a fractional matching with the greatest number of edges $e$ with $f(e)=1$. If there exits some edges with weight $\frac{1}{2}$, then by Lemma \ref{lem-1-4} and Remark \ref{rem-1}, we know $G$ must have $Fig.~2(a)$ as a spanning subgraph. In this case, we can construct a fractional matching with more weight $1$ (see  $Fig.~2(b)$), a contradiction. Therefore, $G$ has a perfect matching .
\end{proof}

In the following, we present a lower bound for the spectral radius which guarantees $\beta(G)\geq \beta+1$, where $\beta~(\beta\leq\frac{n-2}{2})$ is
a positive integer number.
\begin{thm}\label{thm-4}
Let $G$ be a connected graph of order $n~(n\geq3)$. Given a positive integer number $\beta~(\beta\leq\frac{n-2}{2})$, then
\begin{enumerate}[(1)]
\vspace{-0.2cm}
\item For $\frac{n+1}{3}\leq\beta\leq\frac{n-2}{2}~(n\geq8)$, if $\rho(G)>\theta$, where $\theta$ is the largest root of $x^3-(2\beta-2)x+(1-n)x+2(\beta-1)(n-2\beta)=0$ , then $\beta(G)\geq\beta+1$.
\vspace{-0.2cm}
\item For $\beta\leq\frac{n}{3}$, if $\rho(G)>\frac{\beta-1+\sqrt{(\beta-1)^2+4\beta(n-\beta)}}{2}$ , then $\beta(G)\geq\beta+1$.
\end{enumerate}
\end{thm}
\begin{proof}
Let $M$ be the maximum matching of $G$. If $M$ is a perfect matching, it is easy to see $\beta(G)=\frac{n}{2}=\frac{n-2}{2}+1\geq\beta+1$.

If $M$ is a near perfect matching, then $\beta(G)=\frac{n-1}{2}$, $n$ is odd number. Since $\beta$ is an integer number,
$\beta\leq \frac{n-3}{2}$. Then $\beta(G)=\frac{n-1}{2}=\frac{n-3}{2}+1\geq\beta+1$ , as required.

If $M$ is neither perfect matching nor near-perfect matching, then there exist at least two vertices not matched by $M$,
we denote $U$ as the set of vertices which not matched by $M$. Obviously, $|U|\geq2$, and $|U|$ is an independent set.
Suppose $v_1, v_2\in U$, let $G'$ be the graph obtained from $G$ by adding an edge between $v_1$ and $v_2$, it is obviously
that $\beta(G')=\beta(G)+1$ and $\rho(G')>\rho(G)$. If $\beta(G)\leq\beta$,
we can construct a graph $\widehat{G}$ such that $\beta(\widehat{G})=\beta$ by adding edges in the vertex set $U$,
it is easy to see that $\rho(G)\leq\rho(\widehat{G})$ by Lemma \ref{lem-1-1}.  Combining with Theorem \ref{thm-2}, we know
\begin{enumerate}[(1)]
\vspace{-0.2cm}
\item If $2\beta+2\leq n\leq3\beta-1$, then $\rho(G)\leq\rho(\widehat{G})\leq \theta$ , where $\theta$ is the largest root of $x^3-(2\beta-2)x+(1-n)x+2(\beta-1)(n-2\beta)=0$ , a contradiction.
\vspace{-0.4cm}
\item If $n\geq3\beta$, then $\rho(G)\leq\rho(\widehat{G})\leq \frac{\beta-1+\sqrt{(\beta-1)^2+4\beta(n-\beta)}}{2}$ , a contradiction.
\end{enumerate}
This completes the proof.
\end{proof}

 From Theorem \ref{thm-4}, we can also get the main result of O \cite{SO} (see Corollary \ref{cor3-2}) immediately.

\begin{proof}
By Theorem \ref{thm-4}, let $\beta=\frac{n-2}{2}$, if $n\geq8$, $\rho(G)>\theta$, where $\theta$ is the largest root of $x^3-(n-4)x^2-(n-1)x+2(n-4)=0$, then $\beta(G)\geq \frac{n-2}{2}+1=\frac{n}{2}$, which means $G$ has a perfect matching.

When $n=4$, $\beta=\frac{n-2}{2}=1$.  By Theorem \ref{thm-4}, if $\rho(G)>\sqrt{3}$ , then $\beta(G)\geq \frac{n}{2}$, where $\sqrt{3}$ is the largest root of $x^3-3x=0$.

When $n=6$, $\beta=\frac{n-2}{2}=2$. By Theorem \ref{thm-4}, if $\rho(G)>\frac{1+\sqrt{33}}{2}$, then $\beta(G)\geq \frac{n}{2}$.

This completes the proof.
\end{proof}

\end{document}